\documentclass{amsart}
\usepackage[T1]{fontenc}
\usepackage{mathrsfs}
\usepackage[utf8]{inputenc}
\usepackage[italian,english]{babel}

\usepackage{amsthm,amssymb,amsmath}

\usepackage{pgf}
\usepackage{tikz}
\usetikzlibrary{arrows}
\usepackage{hyperref}
\usepackage{float}
\usepackage[labelformat=empty]{caption,subfig}
\usepackage[textmath,displaymath,floats,graphics,tightpage]{preview}

\theoremstyle{definition}
\newtheorem{Def}{Definition}
\newtheorem{step}{Step}
\theoremstyle{plain}
\newtheorem{Thm}{Theorem}

\newtheorem{Cor}{Corollary}

\newtheorem{Prop}{Proposition}
\theoremstyle{remark}
\newtheorem{Rem}{Remark}
\newtheorem{Ex}{Example}
\newtheorem*{op}{Open problem}

\title{Hilbert functions and set of points in $\mathbb P^1\times\mathbb P^1$}

\author{Paola Bonacini} 
\email{bonacini@dmi.unict.it}

\author{Lucia Marino}
\email{lmarino@dmi.unict.it}

\address{Università degli Studi di Catania, 
  Viale A. Doria 6,
95125 Catania,
Italy}


\begin{document}

\maketitle

\begin{abstract}
In this paper we study the problem of classifying the
 Hilbert functions of zero-dimensional schemes in $\mathbb P^1\times \mathbb P^1$. In particular,
 in the main result of the paper we give conditions to determine
 some Hilbert functions of set of points in $\mathbb P^1\times \mathbb P^1$ and we
 describe geometrically these schemes. Moreover, we show that the
 Hilbert functions of these schemes depend only on the distribution of
 the points on a set of
 $(1,0)$ and $(0,1)$-lines. 
\end{abstract}

\section{Introduction}

Given $Q=\mathbb P^1\times \mathbb P^1$, Giuffrida, Maggioni and
Ragusa in \cite{GMR} have investigated zero-dimensional schemes in
$Q$, studying in particular their Hilbert functions, which turn out to
be matrices of integers with infinite entries and with particular
numerical properties. These numerical conditions are
sufficient to characterize the Hilbert functions of arithmetically
Cohen-Macaulay zero-dimensional schemes in $Q$ (see \cite{GMR}) and
by the Hilbert function of an arithmetically
Cohen-Macaulay zero-dimensional scheme it is possible to determine a geometrical
description of the scheme. Other results about the Hilbert functions
of zero-dimensional schemes in $Q$ have been obtained 
for fat points (see \cite{G}, \cite{GVT}, \cite{GVT3}, \cite{GVT2} and
\cite{VT}). In this paper in Theorem \ref{T:4} we give 
numerical conditions to determine Hilbert functions of some 
set of points in $Q$. In particular we describe these
schemes and we show
that any zero-dimensional scheme having in a grid of $(1,0)$ and
$(0,1)$-lines the same configuration of points has the same Hilbert function.

Given a zero-dimensional scheme $X\subset Q$ and a point $P\in X$, in Section \ref{sec:sep}
we look for the Hilbert function of $X\setminus\{P\}$ in relation to the
Hilbert function of $X$, giving a sufficient condition in Corollary
\ref{C:1}. In particular, we show that under this condition there
exists just one separator for $P\in X$ and it has minimal degree (see
\cite{M} and \cite{O}). As a consequence we can partially improve some results given in \cite{BM} on the Hilbert function of
the union of a zero-dimensional scheme $X$ with a particular set of
points of $Q$.  

In Section \ref{main} we prove Theorem \ref{T:4}, in which we give
sufficient conditions to determine some Hilbert functions of set of points in $Q$. The conditions
in Theorem \ref{T:4} are quite technical, but they show a way to
new conditions for a characterization of
Hilbert functions of zero-dimensional schemes in $Q$.

In Example \ref{Ex:2} we give a matrix satisfying some of the
conditions Theorem \ref{T:4} and an application of Theorem \ref{T:4} is given in Example \ref{Ex:1},
while in Example \ref{Ex:0} we show that the conditions of Theorem
\ref{T:4} are not necessary.

\section{Notation}

Let $k$ be an algebraically closed field, let $\mathbb P^1=\mathbb
P^1_k$, let $Q=\mathbb P^1\times \mathbb P^1$ and let
$\mathscr O_Q$ be its structure sheaf. Let us consider the bi-graded
ring $S=H^0_*\mathscr O_Q=\bigoplus_{a,b\ge 0}H^0\mathscr
O_Q(a,b)$. For any sheaf $\mathscr F$ and any $a,b\in \mathbb Z$ we
define $\mathscr F(a,b)=\mathscr F\otimes_{\mathscr O_Q} \mathscr O_Q(a,b)$.

For any bi-graded $S$-module $N$ let $N_{i,j}$ be the component of
degree $(i,j)$. For any $(i_1,j_1)$, $(i_2,j_2)\in \mathbb N^2$ we write $(i_1,j_1)\ge
(i_2,j_2)$ if $i_1\ge i_2$ and $j_1\ge j_2$. Given a 0-dimensional scheme $X\subset Q$, let $I(X)\subset S$ be the
associated saturated ideal and $S(X)=S/I(X)$ the associated graded ring. 
\begin{Def}
  The function $M_X\colon \mathbb Z\times \mathbb Z\rightarrow \mathbb
  N$ defined by: 
\[
M_X(i,j)=\dim_k {S(X)}_{i,j}=(i+1)(j+1)-\dim_k
  {I(X)}_{i,j} 
\]
is called the \emph{Hilbert function} of $X$. The function $M_X$ can be
represented as an infinite matrix with integer entries
$M_X=(M_X(i,j))=(m_{ij})$ called \emph{Hilbert matrix} of $X$. 
\end{Def}

In this paper we denote $M_X(i,j)$ also by $M_X^{(i,j)}$ to simplify
the notation. Note that $M_X(i,j)=0$ for either $i<0$ or $j<0$, so we restrict
ourselves to the range $i\ge 0$ and $j\ge 0$. Moreover, for $i\gg 0$
and $j\gg 0$ $M_X(i,j)=\deg X$.

\begin{Def}
 Given the Hilbert matrix $M_X$ of a zero-dimensional scheme $X\subset Q$, the \emph{first difference of the Hilbert function} of $X$ is the
  matrix $\Delta M_X=(c_{ij})$, where $c_{ij}=m_{ij}-m_{i-1j}-m_{ij-1}+m_{i-1j-1}$.
\end{Def}

We consider the matrices $\Delta^R M_X=(a_{ij})$ and $\Delta^C
M_X=(b_{ij})$, with $a_{ij}=m_{ij}-m_{ij-1}$ and
$b_{ij}=m_{ij}-m_{i-1j}$. Note that for any $i,j\ge 0$:
\begin{equation}   \label{eq:7}
  a_{ij}=\sum_{t=0}^ic_{tj}
\text{\quad and \quad}
b_{ij}=\sum_{t=0}^jc_{it}.
\end{equation}
For any matrix $M$ with infinite entries it is possible
to define in a similar way $\Delta M$, $\Delta^RM$ and $\Delta^CM$. 

\begin{Def}[{\cite[Definition 2.2]{GMR}}]
 Let $M=(m_{ij})$ be a matrix such that $m_{ij}=0$ for $i<0$ and
     $j<0$. We say that $M$ is admissible if $\Delta M=(c_{ij})$
     satisfies the following conditions:
\begin{enumerate}
  \item $c_{ij}\le 1$ and $c_{ij}=0$ for $i\gg 0$ or $j\gg 0$;
\item if $c_{ij}\le 0$, then $c_{rs}\le 0$ for any $(r,s)\ge (i,j)$;
\item for every $(i,j)$ $0\le \sum_{t=0}^j c_{it}\le
  \sum_{t=0}^jc_{i-1t}$ and $0\le \sum_{t=0}^i c_{tj}\le \sum_{t=0}^i c_{tj-1}$.
  \end{enumerate} 
\end{Def}

\begin{Thm}[{\cite[Theorem 2.11]{GMR}}]  \label{T0}
  If $X\subset Q$ is  a $0$-dimensional scheme, then $M_X$ is an
  admissible matrix.
\end{Thm}

If $X\subset Q$ is a zero-dimensional scheme, then $2\le
\operatorname{depth}S(X)\le 3$.
\begin{Def}
 A zero-dimensional scheme $X\subset Q$ is called arithmetically
 Cohen-Macaulay (ACM) if $\operatorname{depth}S(X)=2$.
\end{Def}

\begin{Thm}[{\cite[Theorem 4.1]{GMR}}]
A zero-dimensional scheme  $X\subset Q$ is ACM if and only if
$c_{ij}\ge 0$ for any $(i,j)$.
\end{Thm}

Given an admissible matrix $M$, we define:
\begin{equation}
  \label{eq:6}
  T=\{(i,j)\in \mathbb N\times \mathbb N\mid c_{ij}<0\}.
\end{equation}
Then for any $(i,j)\in T$ we set: 
\begin{equation}
  \label{eq:9}
  I_{ij}=\{0,\dots,-c_{ij}-1\}.
\end{equation}

\begin{Rem} \label{rm}
If $X\subset Q$ is a $0$-dimensional scheme, let us consider $a=\min\{i\in \mathbb
N\mid I(X)_{i,0}\ne 0\}-1$ and $b=\min\{j\in \mathbb
N\mid I(X)_{0,j}\ne 0\}-1$. Then by Theorem \ref{T0} $\Delta M_X$ is zero out of the
rectangle with opposite vertices $(0,0)$ and $(a,b)$, because $c_{a+10}=c_{0b+1}=0$. In this case we
say that $\Delta M_X$ is of size $(a,b)$.
\end{Rem}

Let $X\subset Q$ be a zero-dimensional scheme and let $L$ be a line
defined by a form $l$. Let $J=(I(X),l)$ and let $d=\deg(\operatorname{sat} J)$. Then we call $d$ the number of points
of $X$ on the line $L$ and, by abuse of notation, we define
$d=\#(X\cap L)$. We say that $L$ is disjoint from $X$ if $d=0$. 

For any $i\ge 0$ we set $j(i)=\min\{t\in \mathbb N\mid
m_{it}=m_{it+1}\}$ and similarly for any $j\ge 0$ we set $i(j)=\min\{t\in \mathbb N\mid m_{tj}=m_{t+1j}\}$.

\begin{Thm}[{\cite[Theorem 2.12]{GMR}}]  \label{T}
  Let $X\subset Q$ be a zero-dimensional scheme and let $M_X=(m_{ij})$
  be its Hilbert matrix. Then for every $j\ge 0$ there are just
  $a_{i(0)j}-a_{i(0)j+1}$ lines of type $(1,0)$ each containing just
  $j+1$ points of $X$ and, similarly, for every $i\ge 0$ there are
  just $b_{ij(0)}-b_{i+1j(0)}$ lines of type $(0,1)$ each containing
  just $i+1$ points of $X$.
\end{Thm}

Now we recall the following definition:
\begin{Def}
  Let $X\subset Q$ be a zero-dimensional scheme and let $P\in X$. The
  multiplicity of $X$ in $P$, denoted by $m_X(P)$, is the length of
  $\mathscr O_{X,P}$.
\end{Def}

Given $P\in Q$, we denote by $I_P$ the maximal ideal of $S$ associated
to $P$. If $X\subset Q$ is a $0$-dimensional scheme,  then
$I(X)=\cap_{P'\in X} J_{P'}$ for some ideal $J_{P'}$ such that $\sqrt{J_{P'}}=I_{P'}$.

\begin{Def}
  Given a zero-dimensional scheme $X\subset Q$ and $P\in X$ such that
  $m_X(P)=1$, we say that $f\in S$ is a \emph{separator} for $P\in X$ if $f(P)\ne 0$ and
  $f\in \cap_{P'\in X\setminus \{P\}}J_{P'}$.
\end{Def}

This definition generalizes the definition of a separator for a point
in a reduced zero-dimensional scheme in a multiprojective space given
by \cite{GVT2}.

\section{Separators and Hilbert functions}  \label{sec:sep}

Let $X\subset Q$ be a zero-dimensional scheme and let $M_X$ be its Hilbert
matrix. In all this paper we suppose that $\Delta M_X$ is of size
$(a,b)$ and we denote by $R_0$,\dots, $R_a$
and $C_0$,\dots,$C_b$, respectively, the
$(1,0)$ and $(0,1)$-lines containing $X$ and each one at least one point of
$X$. 

\begin{Thm}  \label{T:0}
  Let $P=R_h\cap C_k\in X$ for some $h\in \{0,\dots,a\}
$  and $k\in \{0,\dots,b\}
$ and suppose that $m_X(P)=1$. Let $Z=X\setminus\{P\}$, $p=\#(Z\cap R_h)$ and $q=\#(Z\cap
C_k)$. If there exists a separator in degree $(q,p)$ for $P\in X$,
then:
 \[
\Delta M_Z^{(i,j)}=
\begin{cases}
 \Delta M_X^{(i,j)} & \text{if } (i,j)\ne (q,p)\\
 \Delta M_X^{(i,j)}-1 & \text{if } (i,j)=(q,p).
\end{cases}
\]
\end{Thm}
\begin{proof}
It is easy to see that  $\Delta M_Z^{(i,j)}=\Delta M_X^{(i,j)}$ for any $(i,j)$ with either
   $i<q$ or $j<p$. Indeed, taken $(i,j)$ with $i< q$ any $(i,j)$-curve containing $Z$ must
contain $C_k$ and so $h^0\mathscr
I_Z(i,j)=h^0\mathscr I_X(i,j)$ and $\Delta M_Z^{(i,j)}=\Delta
M_X^{(i,j)}$. The proof works in a similar way if $j<p$.

By the exact sequence:
\begin{equation}  \label{eq:14}
0\rightarrow \mathscr I_X\rightarrow \mathscr I_Z\rightarrow \mathscr
O_P\rightarrow 0
\end{equation}
we see that $h^0\mathscr I_Z(q,p)>h^0\mathscr I_X(q,p)$ if and
only if $h^0\mathscr I_Z(q,p)=h^0\mathscr I_X(q,p)+1$. This means
that it must be:
\[
\Delta M_Z^{(q,p)}=\Delta M_X^{(q,p)}-1.
\]
Now we only need to prove that  $\Delta M_Z^{(i,j)}=\Delta M_X^{(i,j)}$ for any $(i,j)>(q,p)$.
By \eqref{eq:14} we see that for any $(i,j)$:
\begin{equation}  \label{eq:15}
h^0\mathscr I_X(i,j)\le h^0\mathscr I_Z(i,j)\le h^0\mathscr I_X(i,j)+1
\end{equation}
which is equivalent to:
\[
M_X^{(i,j)}-1\le M_Z^{(i,j)}\le M_X^{(i,j)}.
\]
Since  $h^0\mathscr I_Z(q,p)=h^0\mathscr I_X(q,p)+1$,  
by \eqref{eq:15} we see that it must be $h^0\mathscr I_Z(i,j)=h^0\mathscr I_X(i,j)+1$
for any $(i,j)\ge (q,p)$. In particular this means that
$M_Z^{(i,j)}=M_X^{(i,j)}-1$ for any $(i,j)\ge (q,p)$. Now the conclusion follows easily.
\end{proof}

\begin{Thm}  \label{T:1}
Let $P=R_h\cap C_k\in X$ for some $h\in \{0,\dots,a\}
$  and $k\in \{0,\dots,b\}
$ such that $m_X(P)=1$ and let $p+1=\#(X\cap R_h)$ and $q+1=\#(X\cap
C_k)$. Suppose that one of the following conditions holds:
\begin{enumerate}
\item $p=b$;
\item $q=a$;
\item $p<b$, $q<a$ and $\Delta
  M_X^{(i,j)}=0$ for any $(i,j)\ge (q+1,p+1)$.
\end{enumerate}
Then there exists a separator for $P\in X$ in degree $(q,p)$.
\end{Thm}
\begin{proof}
We divide the proof in different steps. Let $Z=X\setminus\{P\}$.

\begin{step}  \label{s:2}
There exists $\overline j$ with $p\le \overline j\le b$ such that one
the following conditions holds:
\begin{enumerate}
\item  $\Delta M_Z^{(q,j)}=\Delta M_X^{(q,j)}$ for any $j< \overline j$ and
$\Delta M_Z^{(q,\overline j)}<\Delta M_X^{(q,\overline j)}$; 
\item $\Delta M_Z^{(q,j)}=\Delta M_X^{(q,j)}$ for any $p\le j\le b$.
\end{enumerate}
\end{step}

Since $Z\subset X$ we see that $M_Z^{(q,p)}\le
M_X^{(q,p)}$. Moreover, as we have seen in the proof of Theorem
\ref{T:0} $M_Z^{(i,j)}=M_X^{(i,j)}$ for any $i<q$ or $j<p$. This
implies that $\Delta M_Z^{(q,p)}\le \Delta M_X^{(q,p)}$.
If $\Delta M_Z^{(q,p)}=\Delta M_X^{(q,p)}$, then we
can repeat the previous procedure to show that $\Delta
M_Z^{(q,p+1)}\le \Delta M_X^{(q,p+1)}$. By iterating this
procedure we get the conclusion of Step \ref{s:2}.

\begin{step} \label{s:3}
The following equalities hold:
\begin{enumerate}
\item $\sum_{j=p}^b\Delta M_Z^{(q,j)}=\sum_{j=p}^b\Delta M_X^{(q,j)}-1$;
\item for any $i\in \{q+1,\dots,a\}$ $\sum_{j=p}^b\Delta
  M_Z^{(i,j)}=\sum_{j=p}^b\Delta M_X^{(i,j)}$.
\end{enumerate}
\end{step}
Let us first note that by Theorem \ref{T}: 
\[
  b_{q-1j(0)}(Z)-b_{qj(0)}(Z)=
\sum_{j\le b}
\Delta M_Z^{(q-1,j)}-\sum_{j\le b} \Delta M_Z^{(q,j)}
\]
is equal to the number of $(0,1)$-lines containing precisely $q$ points of
$Z$, while: 
\[  
  b_{q-1j(0)}(X)-b_{qj(0)}(X)=
\sum_{j\le b}
\Delta M_X^{(q-1,j)}-\sum_{j\le b} \Delta M_X^{(q,j)}
\]
is equal to the number of $(0,1)$-lines containing precisely $q$ points of
$X$. By hypothesis it must be:
\[
 b_{q-1j(0)}(Z)-b_{qj(0)}(Z)=
\sum_{j\le b}
\Delta M_Z^{(q-1,j)}-\sum_{j\le b} \Delta M_Z^{(q,j)}=\sum_{j\le b}
\Delta M_X^{(q-1,j)}-\sum_{j\le b} \Delta M_X^{(q,j)}+1
\]
Since $h^0\mathscr I_Z(i,j)=h^0\mathscr I_X(i,j)$ for any $i<q$ or
$j<p$, this implies that:
\begin{equation}
  \label{eq:12}
  \sum_{j\le b} \Delta M_Z^{(q,j)}=\sum_{j\le b} \Delta M_X^{(q,j)}-1.
\end{equation}
In a similar way we see that:
\[
 b_{qj(0)}(Z)-b_{q+1j(0)}(Z)=
\sum_{j\le b}
\Delta M_Z^{(q,j)}-\sum_{j\le b} \Delta M_Z^{(q+1,j)}=\sum_{j\le b}
\Delta M_X^{(q,j)}-\sum_{j\le b} \Delta M_X^{(q+1,j)}-1
\]
which implies by \eqref{eq:12} that $\sum_{j\le b} \Delta M_Z^{(q+1,j)}=\sum_{j\le b} \Delta M_X^{(q+1,j)}$.

Let us now suppose that for some $i\ge q+1$, with $i<a$, we have:
\begin{equation}
  \label{eq:4}
  \sum_{j\le b} \Delta M_Z^{(i,j)}=\sum_{i\le b} \Delta M_X^{(i,j)}.
\end{equation}
We will show that:
\begin{equation}
  \label{eq:5}
  \sum_{j\le b} \Delta M_Z^{(i+1,j)}=\sum_{j\le b} \Delta M_X^{(i+1,j)}.
\end{equation}
Again, by Theorem \ref{T} $\sum_{j\le b}
\Delta M_Z^{(i,j)}-\sum_{j\le b} \Delta M_Z^{(i+1,j)}$ is equal
to the number of $(0,1)$-lines containing precisely $i+1$ points of
$Z$, while $\sum_{j\le b}
\Delta M_X^{(i,j)}-\sum_{j\le b} \Delta M_X^{(i+1,j)}$ is equal
to the number of $(0,1)$-lines containing precisely $i+1$ points of
$X$. By hypothesis it must be:
\[
  \sum_{j\le b}
\Delta M_Z^{(i,j)}-\sum_{j\le b} \Delta M_Z^{(i+1,j)}=
\sum_{j\le b}
\Delta M_X^{(i,j)}-\sum_{j\le b} \Delta M_X^{(i+1,j)}.
\]
By \eqref{eq:4} it means that \eqref{eq:5} holds, so that $\sum_{j\le b} \Delta M_Z^{(i,j)}=\sum_{j\le b} \Delta M_X^{(i,j)}$
for any $i$ with $q+1\le i\le a$. 
\\

The statement of the theorem is proved if we show the following:

\begin{step} \label{s:4}
$h^0\mathscr I_Z(q,p)=h^0\mathscr I_X(q,p)+1$.
\end{step}

In the cases $p=b$ and $q=a$ by Step \ref{s:3} we easily get Step
\ref{s:4}. So from now on we suppose that $p<b$ and $q<a$. 

By Step \ref{s:2} and Step \ref{s:3} we see that there exists
$\overline j$ with $p\le \overline j\le b$ such that $\Delta M_Z^{(q,j)}=\Delta M_X^{(q,j)}$
for any $j< \overline j$ and $\Delta M_Z^{(q,\overline j)}<\Delta M_X^{(q,\overline j)}$.
Let us suppose that $\overline j\ge p+1$. Then $\Delta
M_Z^{(q,\overline j)}\le 0$ and by Theorem \ref{T0} we see that $\Delta M_Z^{(i,\overline j)}\le 0$ for any
$i\ge q$. By Step \ref{s:3} and by hypothesis we see that:
\[
\sum_{i=q+1}^a \Delta M_Z^{(i,\overline j)}=\Delta M_X^{(q,\overline j)}-\Delta M_Z^{(q,\overline j)}>0.
\]
So $\sum_{i=q+1}^a \Delta M_Z^{(i,\overline j)}>0$, but this
contradicts that fact that $\Delta M_Z^{(i,\overline j)}\le 0$ for any
$i\ge q$. This means that $\overline j=p$, i.e. $\Delta M_Z^{(q,p)}<\Delta M_X^{(q,p)}$.
Since, as we have seen, $h^0\mathscr I_Z(i,j)=h^0\mathscr
I_X(i,j)$ for any $(i,j)<(q,p)$, it gives us the inequality
$M_Z^{(q,p)}<M_X^{(q,p)}$, which means that $h^0\mathscr I_Z(q,p)>h^0\mathscr I_X(q,p)$. But by the exact sequence:
\[
0\rightarrow \mathscr I_X\rightarrow \mathscr I_Z\rightarrow \mathscr
O_P\rightarrow 0
\]
we see that $h^0\mathscr I_Z(q,p)>h^0\mathscr I_X(q,p)$ if and
only if $h^0\mathscr I_Z(q,p)=h^0\mathscr I_X(q,p)+1$ and the
statement is proved.
\end{proof}

\begin{Cor}   \label{C:1}
  Let $P=R_h\cap C_k\in X$ for some $h\in \{0,\dots,a\}
$  and $k\in \{0,\dots,b\}
$ such that $m_X(P)=1$. Given $Z=X\setminus\{P\}$, $p=\#(Z\cap R_h)$ and $q=\#(Z\cap
C_k)$, suppose that one of the following conditions holds:
\begin{enumerate}
\item $p=b$;
\item $q=a$;
\item $p<b$, $q<a$ and $\Delta
  M_X^{(i,j)}=0$ for any $(i,j)\ge (q+1,p+1)$.
\end{enumerate}
Then:
 \[
\Delta M_Z^{(i,j)}=
\begin{cases}
 \Delta M_X^{(i,j)} & \text{if } (i,j)\ne (q,p)\\
 \Delta M_X^{(i-1,j)}-1 & \text{if } (i,j)=(q,p).
\end{cases}
\]
\end{Cor}
\begin{proof}
  The proof follows by Theorem \ref{T:0} and Theorem \ref{T:1}.
\end{proof}

\begin{Cor}  \label{C:2}
  Let $X$ be an ACM zero-dimensional scheme and let $P=R_h\cap C_k\in X$ for some $h\in \{0,\dots,a\}
$  and $k\in \{0,\dots,b\}
$ such that $m_X(P)=1$. Given $Z=X\setminus\{P\}$, $p=\#(Z\cap R_h)$ and $q=\#(Z\cap
C_k)$, we have:
\[
\Delta M_Z^{(i,j)}=
\begin{cases}
 \Delta M_X^{(i,j)} & \text{if } (i,j)\ne (q,p)\\
 \Delta M_X^{(i-1,j)}-1 & \text{if } (i,j)=(q,p).
\end{cases}
\]
\end{Cor}
\begin{proof}
  By \cite[Proposition 4.1]{BM} we see that $\Delta M_X^{(i,j)}=0$ for
  any $(i,j)\ge (q+1,p+1)$. Then the conclusion follows by Corollary \ref{C:1}.
\end{proof}

In the following we slightly improve the result given in
\cite[Theorem 3.1]{BM}.
\begin{Cor}  
Let $R$ be a $(1,0)$-line disjoint from $X$. Let $C_{b+1}$,\dots,$C_n$,
$n\ge b$, be arbitrary $(0,1)$-lines and $i_1$,\dots,$i_r\in
\{0,\dots,b\}$. Let $\mathcal P=\{R\cap C_i\mid i\in\{0,\dots,n\},\,
i\ne i_1,\dots,i_r\}$ and let $W=X\cup \mathcal P$. Suppose also that on the $(0,1)$-line $C_{i_k}$
there are $q_k$ points of $X$ for $k=1,\dots,r$ and that $q_1\le
q_2\le \dots \le q_r$. Then, given
$T=\{(q_1,n),(q_2,n-1),\dots,(q_r,n-r+1)\}$, we have:
\[
\Delta M_W^{(i,j)}=
\begin{cases}
1 & \text{if } i=0,\, j\le n\\
0 & \text{if }i=0,\, j\ge n+1\\
 \Delta M_X^{(i-1,j)} & \text{if } i\ge 1\text{ and }(i,j)\notin T\\
 \Delta M_X^{(i-1,j)}-1 & \text{if } i\ge 1\text{ and }(i,j)\in T
\end{cases}
\]
if one of the following conditions holds:
\begin{enumerate}
\item $r=1$;
\item $r\ge 2$ and for any $k\in \{2,\dots,r\}$ and $i\ge q_k$\, $\Delta
  M_X^{(i,n-k+2)}=0$.
\end{enumerate}
\end{Cor}
\begin{proof}
Let $Y=X\cup (R\cap (C_0\cup \dots \cup C_n))$. Then the statement
follows by \cite[Lemma 2.15]{GMR} and by Corollary \ref{C:1}.
\end{proof}

\section{Technical results}

In this section we prove some technical results that will be useful in
the proof of Theorem \ref{T:4}. In all this section we denote by $M$
an admissible matrix and we keep the notation given previously. 

\begin{Prop}%
\label{P:1}
  Let us suppose that for some $(i_1,j_1)$ and $(i_2,j_2)$ with
      $j_1>j_2$ the
      following conditions hold:
  \begin{enumerate}
  \item $c_{i_1j_1}<0$ and $c_{i_2j_2}\le 0$;
\item $a_{i_1j_1}+r\ge a_{i_2j_2}$, for some $r\in I_{i_1j_1}$.
  \end{enumerate}
Then $i_1\le i_2$.
\end{Prop}
\begin{proof}
  Let us suppose that $i_1>i_2$. Then by hypothesis we have $\sum_{t=0}^{i_1}c_{tj_1}+r\ge \sum_{t=0}^{i_2}c_{tj_2}$
and so  by Theorem \ref{T0}:
\[
0\ge \sum_{t=0}^{i_2}c_{tj_1}-\sum_{t=0}^{i_2}c_{tj_2}\ge -r-\sum_{t=i_2+1}^{i_1}c_{tj_1}.
\]
This implies that:
\[
0\le r+\sum_{t=i_2+1}^{i_1}c_{tj_1}\le r+c_{i_1j_1}<0,
\]
by hypothesis and by the fact that  by Theorem \ref{T0}
$c_{tj_1}\le 0$ for any $t\ge i_2$.
\end{proof}

In a similar way it is possible to prove the following:
\begin{Prop}  
  Let us suppose that for some $(i_1,j_1)$ and $(i_2,j_2)$ with
  $i_1>i_2$ the
  following conditions hold:
  \begin{enumerate}
  \item $c_{i_1j_1}<0$ and $c_{i_2j_2}\le 0$;
\item $b_{i_1j_1}+r\ge b_{i_2j_2}$, for some $r\in I_{i_1j_1}$.
  \end{enumerate}
Then $j_1\le j_2$.
\end{Prop}

Another technical result is:
\begin{Prop}   \label{P:6}
  Let us suppose that for some  $(i_1,j_1)$
  and $(i_2,j_2)$, with $j_2<j_1-1$, the
  following conditions hold:
  \begin{enumerate}
  \item $c_{i_1j_1}<0$ and $c_{i_2j_2}\le 0$;
\item $a_{i_1j_1}+r\ge a_{i_2j_2}$, for some $r\in I_{i_1j_1}$.
  \end{enumerate}
Then there exists $(i,j)$ with $j_2<j<j_1$ and $i\le i_2$ such that
$c_{ij}<0$ and $a_{i_1j_1}+r+c_{ij}+1\le a_{ij}\le a_{i_1j_1}+r$.
\end{Prop}
\begin{proof}
First note that by Proposition \ref{P:1} it must be $i_1\le i_2$. Suppose
that for every $(i,j)$ with $j_2<j<j_1$ and $i\le i_2$ we have
$c_{ij}\ge 0$. Then this implies that $a_{i_2j_2+1}\ge
a_{i_1-1j_2+1}$, by which we get:
\[
a_{i_2j_2}\ge a_{i_2j_2+1}\ge a_{i_1-1j_2+1}\ge a_{i_1-1j_1}.
\]
However:
\[
a_{i_2j_2}\le a_{i_1j_1}+r=a_{i_1-1j_1}+c_{i_1j_1}+r<a_{i_1-1j_1},
\]
which gives us a contradiction. 

Take $j$ with $j_2<j<j_1$ such that $c_{ij}<0$ for some $i\le
i_2$. Then we can choose $i$ in such a way that $a_{ij}=a_{i_2j}$. Then by Theorem \ref{T0} we see that
$a_{ij}\le a_{i_2j_2}\le a_{i_1j_1}+r$. If $a_{i_1j_1}+r+c_{ij}+1\le a_{ij}$,
then we get the conclusion. So we can suppose that:
\begin{equation}
  \label{eq:3}
  a_{ij}<a_{i_1j_1}+r+c_{ij}+1. 
\end{equation}
Take $i'<i$ such that $c_{i'j}<0$ and $c_{kj}=0$ for
$k=i'+1,\dots,i-1$. Then $a_{ij}=c_{ij}+a_{i'j}$ and \eqref{eq:3} is
equivalent to:
\[
a_{i'j}\le a_{i_1j_1}+r.
\]
Again, if $a_{i_1j_1}+r+c_{i'j}+1\le a_{i'j}$, then the conclusion
follows. Otherwise we proceed as before. Iterating this procedure we
see that either we get the conclusion or $a_{kj}\le a_{i_1j_1}+r$ for
$k$ such that $c_{kj}=1$ and $c_{k+1j}\le 0$. So we can suppose that
such a $k$ exists. Then we see that $a_{kj}=\max\{a_{ij}\mid i\ge
0\}\ge a_{i_1-1j}$, so that $a_{i_1-1j}\le a_{i_1j_1}+r<a_{i_1-1j_1}$. But
by Theorem \ref{T0} this is not possible. 
\end{proof}

In a similar way it is possible to prove the following:
\begin{Prop}  \label{P:8}
  Let us suppose that for some  $(i_1,j_1)$
  and $(i_2,j_2)$, with $i_2<i_1-1$, the
  following conditions hold:
  \begin{enumerate}
  \item $c_{i_1j_1}<0$ and $c_{i_2j_2}\le 0$;
\item $b_{i_1j_1}+r\ge b_{i_2j_2}$, for some $r\in I_{i_1j_1}$.
  \end{enumerate}
Then there exists $(i,j)$ with $i_2<i<i_1$ and $j\le j_2$ such that
$c_{ij}<0$ and $b_{i_1j_1}+r+c_{ij}+1\le b_{ij}\le b_{i_1j_1}+r$.
\end{Prop}

\begin{Rem}  \label{r:1}
  By Proposition \ref{P:6} it follows that, given $(i_1,j_1)$,
  $(i_2,j_2)\in T$, $r_1\in I_{i_1j_1}$ and $r_2\in I_{i_2j_2}$ such that $a_{i_1j_1}+r_1=a_{i_2j_2}+r_2$, for
  any $j$ with $j_1\le j\le j_2$ there exists $i$ such that $(i,j)\in
  T$ and $a_{ij}+r=a_{i_1j_1}+r_1=a_{i_2j_2}+r_2$ for some $r\in I_{ij}$. 

Of course, a similar result follows by Proposition \ref{P:8}.
\end{Rem}

Now we prove a result on $\Delta^R M$.
\begin{Prop}   \label{P:7}
  Let $(i_1,j_1)$, $(i_2,j_1)\in T$ with $i_2<i_1$. Then
  $a_{i_2j_1}+s>a_{i_1j_1}+r$,  for any $r\in I_{i_1j_1}$
  and $s\in I_{i_2j_1}$.
\end{Prop}
\begin{proof}
  Let us suppose that $a_{i_2j_1}+s\le a_{i_1j_1}+r$. Note that
  $a_{i_1j_1}=a_{i_2j_1}+\sum_{i=i_2+1}^{i_1}c_{ij_1}$. Then we have:
\begin{equation} \label{eq:1}
s\le \sum_{i=i_2+1}^{i_1}c_{ij_1}+r.
\end{equation}
However $c_{i_1j_1}+r<0$ and by Theorem \ref{T0} $c_{ij_1}\le
0$ for any $i>i_2$. Then by \eqref{eq:1} we get $s<0$, which gives us
a contradiction.
\end{proof}

In a similar way it is possible to prove the following:
\begin{Prop}  \label{P:10}
  Let $(i_1,j_1)$, $(i_1,j_2)\in T$ with $j_2<j_1$. Then
  $b_{i_1j_2}+s>b_{i_1j_1}+r$,  for any $r\in I_{i_1j_1}$
  and $s\in I_{i_1j_2}$.
\end{Prop}

Given the admissible matrix $M$ of size $(a,b)$, let us
consider $R_0$,\dots,$R_a$ and $C_0$,\dots,$C_b$ pairwise distinct arbitrary $(1,0)$ and
$(0,1)$-lines. Let $P_{ij}=R_i\cap C_j$ and let us consider the
following reduced ACM zero-dimensional
scheme:
\[
X=\{P_{ij}\mid c_{ij}=1\}.
\]
Under this notation we prove the following:
\begin{Prop}   \label{P:3}
Let $p\in \mathbb N$ sucht that:
\[
\{(i,j)\in T\mid p+c_{ij}+1\le a_{ij}\le p\}\ne \emptyset
\]
and let:
\[
k=\max\{j\mid \exists\, (i,j)\in T,\,  p+c_{ij}+1\le a_{ij}\le
p\}. 
\]
Then $0\le p\le a$ and $\#(X\cap R_{p})=k+1$.
\end{Prop}
\begin{proof}
Let $(h,k)\in T$ such that $p+c_{hk}+1\le a_{hk}\le p$. Then there exists $s\in I_{hk}$ such that
$a_{hk}+s=p$. This implies that $0\le p\le a$.

Now we prove that $\#(X\cap R_{p})=k+1$. We will show that $c_{pk}=1$
and $c_{pk+1}\le 0$. Let us first note that:
\[
p=a_{hk}+s=a_{h-1k}+c_{hk}+s\le h-1<h.
\]
Let us suppose now that $c_{pk}\le 0$. In this case by \eqref{eq:7} we see that:
\[
a_{hk}+s=a_{h-1k}+c_{hk}+s<a_{hk-1}\le a_{p-1k}\le p,
\]
which contradicts the fact that $a_{hk}+s=p$. So we can say that $c_{pk}=1$.

Let us suppose now that $c_{pk+1}=1$. Then by \eqref{eq:7} we get:
\[
a_{pk+1}=p+1=a_{hk}+s+1.
\]
By Theorem \ref{T0} we see that $a_{hk}\ge a_{hk+1}$ and we
also have $a_{hk}<a_{hk}+s+1=a_{pk+1}$. This implies that $a_{hk+1}< a_{pk+1}$, but $p<h$ and so there exists $i$ with $p<i\le h$ such that
$c_{ik+1}<0$. Let $i\le h$ such that $c_{ik+1}<0$ and
$a_{ik+1}=a_{hk+1}\le a_{hk}$. By hypothesis on $k$ it must be
$a_{ik+1}<p+c_{ik+1}+1$. So, taken $i'$ such that
$c_{i'k+1}<0$ and 
$c_{i'+1k+1}=\dots=c_{h-1k+1}=0$, we see that
$a_{i'k+1}\le p$. Again, by hypothesis it must be
$a_{i'k+1}<p+c_{i'k+1}+1$. Iterating the procedure we see that,
taken $m$ such that $c_{mk+1}=1$ and $c_{m+1k+1}\le 0$, it must be
$a_{mk+1}\le p$, where by \eqref{eq:7} $a_{mk+1}=m+1$. However,
$c_{pk+1}=1$ and so $m\ge p$ and so
this gives us a contradiction.  
\end{proof}

In a similar way it is possible to prove the following:
\begin{Prop}   \label{P:9}
Let $q\in \mathbb N$ such that:
\[
\{(i,j)\in T\mid q+c_{ij}+1\le a_{ij}\le q\}\ne \emptyset
\]
and let:
\[
h=\max\{i\mid \exists\, (i,j)\in T,\, q+c_{ij}+1\le b_{ij}\le
  q\}. 
\]
Then $0\le q\le b$ and $\#(X\cap C_{q})=h+1$.
\end{Prop}

\section{Main Theorem}   \label{main}

In this section we give some conditions for an admissible
matrix to be the Hilbert matrix of some reduced zero-dimensional schemes. If $M$
is an admissible matrix of size $(a,b)$, it is always possible to
associate to $M$ a reduced zero-dimensional scheme $Z$ in the following way. Let $R_0$,\dots,$R_a$ and $C_0$,\dots,$C_b$ be pairwise distinct arbitrary $(1,0)$ and
$(0,1)$-lines. Let $P_{ij}=R_i\cap C_j$ and let us consider the scheme:
\[
X=\{P_{ij}\mid c_{ij}=1\}.
\]
By proceeding as in \cite[Proposition 4.1]{BM} we see that $X$ is an ACM
zero-dimensional scheme and
that:
\begin{equation}
  \label{eq:11}
  \Delta M_X^{(i,j)}=
  \begin{cases}
  1 & \text{if }(i,j)\in X\\
  0 & \text{if }(i,j)\notin X.  
  \end{cases}
\end{equation}
Note that $(a_{ij}+r,b_{ij}+r)\in X$ for any $(i,j)\in T$ and $r\in
I_{ij}$ (see \eqref{eq:6} and \eqref{eq:9}). Then it is easy to see that:
\[
\mathcal P=\{P_{a_{ij}+r,b_{ij}+r}\mid (i,j)\in T,\, r\in I_{ij}\}\subsetneq X.
\]

\begin{Def}  \label{d}
  The scheme $Z=X\setminus \mathcal P$ is called \emph{zero-dimensional
  scheme associated to $M$}.
\end{Def}

We call $Z$ the We want to show under which conditions
the Hilbert matrix of $Z$ is $M$. For this purpose we give the
following definitions:
\begin{Def}
  Let $M$ be an admissible matrix. We say that $M$ is a
  \emph{$\Delta$-regular matrix} if for any $(i_1,j_1)$, 
  \dots, $(i_n,j_n)\in T$ and $r_1\in I_{i_1j_1}$, 
  \dots, $r_n\in I_{i_nj_n}$ the following conditions hold:
\begin{enumerate}
  \item  if $a_{i_1j_1}+r_1=\dots=a_{i_nj_n}+r_n$, $i_1\ne
    \dots \ne i_n$ and $j_1<\dots <j_n$, then
   $b_{i_1j_1}+r_1\le \dots \le b_{i_nj_n}+r_n$;
\item  if $b_{i_1j_1}+r_1=\dots=b_{i_nj_n}+r_n$, $j_1\ne \dots
  \ne j_n$ and $i_1<\dots <i_n$, then
    $a_{i_1j_1}+r_1\le \dots \le a_{i_nj_n}+r_n$.
  \end{enumerate}
\end{Def}

\begin{Rem}
Given an admissible matrix $M$ and any  $(i_1,j_1)$, 
  \dots, $(i_n,j_n)\in T$ and $r_1\in I_{i_1j_1}$, 
  \dots, $r_n\in I_{i_nj_n}$ such that $a_{i_1j_1}+r_1=\dots=a_{i_nj_n}+r_n$, $i_1=
    \dots =i_n$ and $j_1<\dots <j_n$, then by Proposition \ref{P:10}
    it must be
   $b_{i_1j_1}+r_1>\dots >b_{i_nj_n}+r_n$.

  Similarly, if $b_{i_1j_1}+r_1=\dots=b_{i_nj_n}+r_n$, $j_1=\dots
   =j_n$ and $i_1<\dots <i_n$, then by Proposition \ref{P:7}
    $a_{i_1j_1}+r_1> \dots >a_{i_nj_n}+r_n$.
\end{Rem}

\begin{Def}
  An admissible matrix $M$ is called \emph{plain
    matrix} if for any $(i_1,j_1)$, $(i_2,j_2)\in T$, $r_1\in
  I_{i_1j_1}$, $r_2\in I_{i_2j_2}$ we have $(a_{i_1j_1}+r_1,b_{i_1j_1}+r_1)\ne (a_{i_2j_2}+r_2,b_{i_2j_2}+r_2)$.
\end{Def}

Note that, if $M$ is plain, then for
any $(i_1,j_1),(i_2,j_2)\in T$, $r_1\in I_{i_1j_1}$ and $r_2\in
I_{i_2j_2}$ we have $P_{a_{i_1j_1}+r_1,b_{i_1j_1}+r_1}\ne
P_{a_{i_2j_2}+r_2,b_{i_2j_2}+r_2}$. 

\begin{Ex}  \label{Ex:2}
Let us consider the following admissible matrix $M$ and its first difference $\Delta M$.
\begin{figure}[H]
\begin{preview}
\begin{center}
\subfloat[$M$]{
\begin{tikzpicture}[x=0.45cm,y=0.45cm,font=\tiny]
\clip(0,0.5) rectangle (11,8);
  \draw[style=help lines,xstep=1,ystep=1] (1,1) grid (10,7);
\foreach \x in {1,...,9} \draw (\x,1) +(.5,.5)  node {\dots};
\foreach \y in {2,...,6} \draw (9,\y) +(.5,.5) node {\dots};
\draw (1.5,2.5) node{$4$};
\draw (2.5,2.5) node{$8$};
\foreach \x in {3,...,8} \draw (\x,2.5) +(.5,0) node {$12$};
\draw (1.5,3.5) node{$4$};
\draw (2.5,3.5) node{$8$};
\foreach \x in {3,...,8} \draw (\x,3.5) +(.5,0) node {$12$};
\draw (1.5,4.5) node{$3$};
\draw (2.5,4.5) node{$6$};
\draw (3.5,4.5) node{$9$};
\foreach \x in {4,...,8} \draw (\x,4.5) +(.5,0) node {$12$};
\draw (1.5,5.5) node{$2$};
\draw (2.5,5.5) node{$4$};
\draw (3.5,5.5) node{$6$};
\draw (4.5,5.5) node{$8$};
\draw (5.5,5.5) node{$10$};
\foreach \x in {6,...,8} \draw (\x,5.5) +(.5,0) node {$11$};
\foreach \x in {1,...,7} \draw (\x,6.5) +(.5,0) node {$\x$};
\draw (8.5,6.5) node {$7$};
\foreach \x in {0,...,8} \draw (\x,7.5) +(1.5,0) node {$\x$};
\foreach \y in {0,...,5} \draw (0.5,6.5-\y) node {$\y$};
\end{tikzpicture}
}
\hspace{1cm}
\subfloat[$\Delta M$]{
\begin{tikzpicture}[x=0.45cm,y=0.45cm,font=\tiny]
\clip(0,0.5) rectangle (11,8);
  \draw[style=help lines,xstep=1,ystep=1] (1,1) grid (10,7);
\foreach \x in {1,...,9} \draw (\x,1) +(.5,.5)  node {\dots};
\foreach \y in {2,...,6} \draw (9,\y) +(.5,.5) node {\dots};
\foreach \x in {1,...,8} \draw (\x,2) +(.5,.5)  node {$0$};
\foreach \y in {2,...,6} \draw (8,\y) +(.5,.5) node {$0$};
\foreach \x in {1,2,3} \draw (\x,3.5) +(.5,0) node{$1$};
\draw (4.5,3.5) node {$-3$};
\foreach \x in {5,...,7} \draw (\x,3.5) +(.5,0) node{$0$};
\foreach \x in {1,2,3,4} \draw (\x,4.5) +(.5,0) node{$1$};
\draw (5.5,4.5) node {$-2$};
\draw (6.6,4.5) node {$-1$};
\draw (7.5,4.5) node {$0$};
\foreach \x in {1,...,5} \draw (\x,5.5) +(.5,0) node{$1$};
\draw (6.5,5.5) node {$0$};
\draw (7.5,5.5) node {$-1$};
\foreach \x in {1,...,7} \draw (\x,6.5) +(.5,0) node{$1$};
\foreach \x in {0,...,8} \draw (\x,7.5) +(1.5,0) node {$\x$};
\foreach \y in {0,...,5} \draw (0.5,6.5-\y) node {$\y$};
\end{tikzpicture}
}
\end{center}
\end{preview}
\end{figure}
It is possible to see that $M$ is $\Delta$-regular and plain. Indeed,
note that $T=\{(1,6),(2,5),(2,4),(3,3)\}$ and that:
\begin{itemize}
\item $c_{16}=-1$, so that $r=0$ and $(a_{16},b_{16})=(0,4)$;
\item $c_{25}=-1$, so that $r=0$ and $(a_{25},b_{25})=(0,1)$;
\item $c_{24}=-2$, so that $r=0,1$, $(a_{24},b_{24})=(0,2)$ and
  $(a_{24}+1,b_{24}+1)=(1,3)$;
\item $c_{33}=-3$, so that $r=0,1,2$, $(a_{33},b_{33})=(0,0)$ and
  $(a_{33}+1,b_{33}+1)=(1,1)$ and $(a_{33}+2,b_{33}+2)=(2,2)$.
\end{itemize}
Since all these pairs are distinct, $M$ is plain. It is
$\Delta$-regular because:
\begin{itemize}
\item taken $(3,3),(2,5),(1,6)\in T$, we get $a_{33}=a_{25}=a_{16}=0$
  and $b_{33}=0<b_{25}=1<b_{16}=4$;
\item taken $(3,3),(2,4),(1,6)\in T$, we get $a_{33}=a_{24}=a_{16}=0$
  and $b_{33}=0<b_{24}=2<b_{16}=4$;
\item taken $(3,3),(2,4)\in T$, we get $a_{33}+1=a_{24}+1=1$ and
  $b_{33}+1=1<b_{24}+1=3$;
\item taken  $(2,5),(3,3)\in T$, we get $b_{25}=b_{33}+1=1$ and
  $a_{25}=0<a_{33}+1=1$;
\item taken $(2,4),(3,3)\in T$, we get $b_{24}=b_{33}+2=2$ and
  $a_{24}=0<a_{33}+2=2$.
\end{itemize}
\end{Ex}

Recalling Definition \ref{d}, we prove the following:

\begin{Thm}   \label{T:4}
  Let $M$ be a plain and $\Delta$-regular matrix such that one of the following conditions holds:
\begin{enumerate}
\item $a_{ij}\ge a_{i-1j+1}$ for any $i,j\ge 0$; 
\item $b_{ij}\ge b_{i+1j-1}$ for any $i,j\ge 0$.
\end{enumerate}
Then $M^{(i,j)}=M_Z^{(i,j)}$ for any $(i,j)$.
\end{Thm}
\begin{proof}
Let us suppose that $b_{ij}\ge b_{i+1j-1}$ for any $i,j\ge 0$. Under
this hypothesis we have that for any $(i_1,j),(i_2,j)\in T$ with
$i_1>i_2$ and for any $r_1\in I_{i_1j}$ and $r_2\in I_{i_2j}$ it is
$b_{i_2j}+r_2>b_{i_1j}+r_1$. Indeed, it is sufficient to show that
$b_{i_2j}>b_{i_1j}-c_{i_1j}-1=b_{i_1j-1}-1$. By hypothesis and by the
fact that $M$ is admissible we have:
\[
b_{i_2j}\ge b_{i_2+1j-1}\ge b_{i_1j-1}>b_{i_1j-1}-1.
\]

Now we will prove that $\Delta M_Z^{(i,j)}=\Delta M^{(i,j)}$ for any
  $(i,j)$. We apply Corollary \ref{C:1} by deleting one by one the points
  $(a_{ij}+r,b_{ij}+r)$, that are all distinct since $M$ is
  plain. We proceed in the following way: given
  $(a_{i_1j_1}+r_1,b_{i_1j_1}+r_1)$ and
  $(a_{i_2j_2}+r_2,b_{i_2j_2}+r_2)$, we delete first
  $(a_{i_2j_2}+r_2,b_{i_2j_2}+r_2)$ if either
  $a_{i_1j_1}+r_1<a_{i_2j_2}+r_2$ or $a_{i_1j_1}+r_1=a_{i_2j_2}+r_2$
  and $b_{i_1j_1}+r_1<b_{i_2j_2}+r_2$. 

Let us first show that it is possible to compute $M_Z$ by applying
recursively Corollary \ref{C:1}. Given the point $(a_{ij}+r,b_{ij}+r)$,
with $c_{ij}<0$ and $r\in I_{ij}$, by what we have just proved and  by
the fact that $M$ is $\Delta$-regular we see that:
\[
\{(h,j)\mid b_{hj}+s=b_{ij}+r,\, h>i,\, s\in I_{hj}\}=\emptyset
\]
and
\begin{equation}
  \label{eq:2}
  \min\{h\mid (h,k)\in T,\,
b_{hk}+s=b_{ij}+r,\, a_{hk}+s\ge
a_{ij}+r,\, s\in I_{hk}\}=i.
\end{equation}
So, keeping the notation of Corollary \ref{C:1} and \eqref{eq:2} together with Remark
\ref{r:1} and  Proposition \ref{P:9} imply:
\begin{equation} 
\label{eq:8}
  q=\#(X\cap C_{b_{ij}+r})-\#\{(h,k)\in T\mid
b_{hk}+s=b_{ij}+r,\, a_{hk}+s\ge a_{ij}+r,\, s\in I_{hk}\}=i.
\end{equation}
Let:
\begin{align*}
m_{ij}^{(r)}=&\min\{k\mid \exists\, (i,k)\in T,\, k\le j,\,
a_{ik}+s=a_{ij}+r,\,s\in I_{ik}\},\\
n_{ij}^{(r)}=&\max\{k\mid \exists\, (i,k)\in T,\, k\ge j\, ,
a_{ik}+s=a_{ij}+r,\, s\in I_{ik}\},\\
p_{ij}^{(r)}=&\, m_{ij}^{(r)}
+\#\{(i,k)\mid a_{ik}+s=a_{ij}+r,\, k>j,\, s\in I_{ik}\}.
\end{align*}
Note that by Remark \ref{r:1} and by Proposition \ref{P:10}:
\begin{equation}  \label{eq:10}
p_{ij}^{(r)}=m_{ij}^{(r)}+n_{ij}^{(r)}-j.
\end{equation}
By by
the fact that $M$ is $\Delta$-regular, by Remark \ref{r:1} and by Propositions \ref{P:10} and \ref{P:3}:
\[
\#(X\cap R_{a_{ij}+r})
-\#\{(h,k)\in T\mid
a_{hk}+s=a_{ij}+r,\, b_{hk}+s\ge b_{ij}+r,\, s\in I_{hk}\}=p_{ij}^{(r)}
\]
that, in the notation of Corollary \ref{C:1} and together with
\eqref{eq:8}, gives that: 
\begin{equation}
  \label{eq:16}
  (q,p)=(i,p_{ij}^{(r)}).
\end{equation}

Suppose that the first point to be deleted is $(a_{i_1j_1}+r_1,b_{i_1j_1}+r_1)$ and let $X'=X\setminus
\{(a_{i_1j_1}+r_1,b_{i_1j_1}+r_1)\}$. Then by the fact that $X$ is ACM
we can apply Corollary \ref{C:2}:
\[
\Delta M_{X'}^{(i,j)}=
\begin{cases}
  \Delta M_X^{(i,j)} & \text{for } (i,j)\ne (i_1,p_{i_1j_1}^{(r_1)})\\
\Delta M_X^{(i,j)}-1 & \text{for } (i,j)=(i_1,p_{i_1j_1}^{(r_1)}).
\end{cases}
\]

Iterating the procedure, taken a point
$(a_{i_1j_1}+r_1,b_{i_1j_1}+r_1)$, let us consider:
\begin{multline*}
G=\{(a_{ij}+r,b_{ij}+r)\in \mathcal P\mid
a_{ij}+r>a_{i_1j_1}+r_1\}\cup\\ 
\cup \{(a_{ij}+r,b_{ij}+r)\in \mathcal P\mid
a_{ij}+r=a_{i_1j_1}+r_1,\, b_{ij}+r>b_{i_1j_1}+r_1\}
\end{multline*}
and the correspondent set:
\[
H=\{(i,p_{ij}^{(r)})\mid (a_{ij}+r,b_{ij}+r)\in G\}.
\]
If $X''=X\setminus G$, suppose that we can apply Corollary \ref{C:1}
to the scheme $X''$ by deleting one by one all the points
$(a_{ij}+r,b_{ij}+r)\in G$. In this way we see that:
\begin{equation}
  \label{eq:17}
  \Delta M_{X''}^{(i,j)}=
\begin{cases}
  \Delta M_X^{(i,j)}& \text{if }(i,j)\notin H\\
  \Delta M_X^{(i,j)}-1& \text{if }(i,j)\in H.
\end{cases}
\end{equation}
We will show that we can apply
  Corollary \ref{C:1} to scheme $X'''=X''\setminus
  \{(a_{i_1j_1}+r_1,b_{i_1j_1}+r_1)\}$.

By \eqref{eq:11}, \eqref{eq:16} and by \eqref{eq:17} we know $\Delta M_{X''}^{(i,j)}<0$ if and only if
$(i,j)\in H$. By \eqref{eq:16} we cannot apply
  Corollary \ref{C:1} to $X'''$ if $(i_1+1,p_{i_1j_1}^{(r_1)}+1)\le
  (i_2,p_{i_2j_2}^{(r_2)})$ for some $(i_2,p_{i_2j_2}^{(r_2)})\in
  H$. Since $m_{ij}^{(r)}\le p_{ij}^{(r)}\le n_{ij}^{(r)}$ for every
  $(i,j)\in T$, by Remark \ref{r:1} we have that
  $(i_1,p_{i_1j_1}^{(r_1)}),(i_2,p_{i_2j_2}^{(r_2)})\in T$ and that:
\[
a_{i_1p_{i_1j_1}^{(r_1)}}+s_1=a_{i_1j_1}+r_1\text{\quad and \quad} a_{i_2p_{i_2j_2}^{(r_2)}}+s_2=a_{i_2j_2}+r_2,
\]
for some $s_1\in I_{i_1p_{i_1j_1}^{(r_1)}}$ and $s_2\in I_{i_2p_{i_2j_2}^{(r_2)}}$. 
This means that:
\[
a_{i_1p_{i_1j_1}^{(r_1)}}+s_1=a_{i_1j_1}+r_1\le a_{i_2j_2}+r_2=a_{i_2p_{i_2j_2}^{(r_2)}}+s_2
\]
where $i_1<i_2$ and $p_{i_1j_1}^{(r_1)}<p_{i_2j_2}^{(r_2)}$, which
contradicts Proposition \ref{P:1}. So we can apply Corollary \ref{C:1} and we
  see that:
\[
\Delta M_{X'''}^{(i,j)}=
\begin{cases}
  \Delta M_{X''}^{(i,j)} & \text{for } (i,j)\ne (i_1,p_{i_1j_1}^{(r_1)})\\
\Delta M_{X''}^{(i,j)}-1 & \text{for } (i,j)=(i_1,p_{i_1j_1}^{(r_1)}).
\end{cases}
\]
By iterating the procedure we are able to compute $M_Z$.

Now, note that, taken $(i_1,p_{i_1j_1}^{(r_1)})$ and taken $s_1\in
I_{i_1p_{i_1j_1}^{(r_1)}}$ such that
$a_{i_1p_{i_1j_1}^{(r_1)}}+s_1=a_{i_1j_1}+r_1$, it is easy to see that:
\[
m_{i_1p_{i_1j_1}^{(r_1)}}^{(s_1)}=m_{i_1j_1}^{(r_1)}
\text{\quad and \quad}
n_{i_1p_{i_1j_1}^{(r_1)}}^{(s_1)}=n_{i_1j_1}^{(r_1)}.
\]
This implies together with \eqref{eq:10} that
$p_{p_{i_1j_1}^{(r_1)}}^{(s_1)}=m_{i_1p_{i_1j_1}^{(r_1)}}^{(s_1)}+n_{i_1p_{i_1j_1}^{(r_1)}}^{(s_1)}-p_{i_1j_1}^{(r_1)}=j$. This
means that $\Delta
M_Z^{(i,j)}=\Delta M^{(i,j)}$ for any $(i,j)$. 

The proof works in a similar way if $a_{ij}\ge a_{i-1j+1}$ for any
$i,j\ge 0$.
\end{proof}

\begin{Cor}  \label{C:6}
  Let $M$ be an admissible matrix such that:
\[
a_{i_1j_1}-b_{i_1j_1}<a_{i_2j_2}-b_{i_2j_2}
\]
for any $(i_1,j_1)$, $(i_2,j_2)\in T$, with $i_1<i_2$ and
$j_1>j_2$. Suppose that one of the following conditions holds:
\begin{enumerate}
\item $a_{ij}\ge a_{i-1j+1}$ for any $i,j\ge 0$; 
\item $b_{ij}\ge b_{i+1j-1}$ for any $i,j\ge 0$.
\end{enumerate}
Then $M^{(i,j)}=M_Z^{(i,j)}$ for any $(i,j)$.
\end{Cor}
\begin{proof}
  If $a_{i_1j_1}-b_{i_1j_1}<a_{i_2j_2}-b_{i_2j_2}$
for any $(i_1,j_1)$, $(i_2,j_2)\in T$, with $i_1<i_2$ and
$j_1>j_2$, then $M$ is plain and $\Delta$-regular. Then the statement
follows by Theorem \ref{T:4}.
\end{proof}

\begin{Cor}  \label{C:7}
  Let $M$ be a plain matrix and let $T=\{(i_1,j_1),\dots,(i_n,j_n)\}$.
  If $i_1+j_1=\dots=i_n+j_n$, then $M^{(i,j)}=M_Z^{(i,j)}$ for any $(i,j)$.
\end{Cor}
\begin{proof}
We want to prove that $M$ satisfies the hypothesis of Theorem \ref{T:4}.

By Proposition \ref{P:6}, Proposition \ref{P:8} and by hypothesis for
any $i$ there exists at most one $j\in \mathbb N$ such that $c_{ij}<0$
and, similarly, for any $j$ there exists at most one $i\in \mathbb N$
such that $c_{ij}<0$. Moreover, if $(i,j),(i+k,j-k)\in T$ for some
$i,j,k\in\mathbb N$, then $(i+1,j-1)$, \dots, $(i+k-1,j-k+1)\in T$.   

Now we show that $a_{ij}\ge a_{i-1j+1}$ for any $i,j$. If $c_{ij}=1$,
then this is true because $M$ is an admissible matrix and
$a_{ij}=\sum_{k\le i}c_{kj}$. If $c_{ij}=0$, by the fact that $M$ is admissible:
\[
a_{ij}=a_{i-1j}\ge a_{i-1j+1}.
\]
If $c_{ij}<0$, then $c_{ij+1}=0$ and by the fact that $M$ is
admissible:
\[
a_{ij}\ge a_{ij+1}=a_{i-1j+1}+c_{ij+1}=a_{i-1j+1}.
\]
In a similar way it is possible to see that $b_{ij}\ge b_{i+1j-1}$. 

Now we need to prove that $M$ is $\Delta$-regular. It is sufficient 
to show that for any $(i,j)$, $(i+1,j-1)\in T$:
\[
b_{i+1j-1}-a_{i+1j-1}\le b_{ij}-a_{ij}.
\]
This holds because we have just proved that $b_{i+1j-1}\le b_{ij}$ and
$a_{i+1j-1}\ge a_{ij}$. 
\end{proof}

In the following example we give an application of Theorem \ref{T:4}.
\begin{Ex}  \label{Ex:1}
Given the following matrix $M$, it is easy to see that it satisfies the hypotheses of Theorem \ref{T:4}
and that its first difference $\Delta M$ is the following: 
  \begin{figure}[H]
\begin{preview}
\begin{center}
\subfloat[$M$]{
\begin{tikzpicture}[x=0.45cm,y=0.45cm,font=\tiny]
\clip(0,0.5) rectangle (12,13);
  \draw[style=help lines,xstep=1,ystep=1] (1,1) grid (11,12);
\foreach \x in {1,...,10} \draw (\x,1) +(.5,.5)  node {\dots};
\foreach \y in {2,...,11} \draw (10,\y) +(.5,.5) node {\dots};
\draw (1.5,2.5) node{$9$};
\draw (2.5,2.5) node{$18$};
\draw (3.5,2.5) node{$21$};
\draw (4.5,2.5) node{$23$};
\foreach \x in {5,...,9} \draw (\x,2.5) +(.5,0) node {$24$};
\draw (1.5,3.5) node{$9$};
\draw (2.5,3.5) node{$18$};
\draw (3.5,3.5) node{$21$};
\draw (4.5,3.5) node{$23$};
\foreach \x in {5,...,9} \draw (\x,3.5) +(.5,0) node {$24$};
\draw (1.5,4.5) node{$8$};
\draw (2.5,4.5) node{$16$};
\draw (3.5,4.5) node{$21$};
\draw (4.5,4.5) node{$23$};
\foreach \x in {5,...,9} \draw (\x,4.5) +(.5,0) node {$24$};
\draw (1.5,5.5) node{$7$};
\draw (2.5,5.5) node{$14$};
\draw (3.5,5.5) node{$21$};
\draw (4.5,5.5) node{$23$};
\foreach \x in {5,...,9} \draw (\x,5.5) +(.5,0) node {$24$};
\draw (1.5,6.5) node{$6$};
\draw (2.5,6.5) node{$12$};
\draw (3.5,6.5) node{$18$};
\draw (4.5,6.5) node{$23$};
\foreach \x in {5,...,9} \draw (\x,6.5) +(.5,0) node {$24$};
\draw (1.5,7.5) node{$5$};
\draw (2.5,7.5) node{$10$};
\draw (3.5,7.5) node{$15$};
\draw (4.5,7.5) node{$20$};
\draw (5.5,7.5) node{$23$};
\foreach \x in {6,...,9} \draw (\x,7.5) +(.5,0) node {$24$};
\draw (1.5,8.5) node{$4$};
\draw (2.5,8.5) node{$8$};
\draw (3.5,8.5) node{$12$};
\draw (4.5,8.5) node{$16$};
\draw (5.5,8.5) node{$19$};
\foreach \x in {6,...,9} \draw (\x,8.5) +(.5,0) node {$22$};
\draw (1.5,9.5) node{$3$};
\draw (2.5,9.5) node{$6$};
\draw (3.5,9.5) node{$9$};
\draw (4.5,9.5) node{$12$};
\draw (5.5,9.5) node{$15$};
\draw (6.5,9.5) node{$18$};
\foreach \x in {7,...,9} \draw (\x,9.5) +(.5,0) node {$19$};
\draw (1.5,10.5) node{$2$};
\draw (2.5,10.5) node{$4$};
\draw (3.5,10.5) node{$6$};
\draw (4.5,10.5) node{$8$};
\draw (5.5,10.5) node{$10$};
\draw (6.5,10.5) node{$12$};
\draw (7.5,10.5) node{$13$};
\foreach \x in {8,9} \draw (\x,10.5) +(.5,0) node {$14$};
\draw (1.5,11.5) node{$1$};
\draw (2.5,11.5) node{$2$};
\draw (3.5,11.5) node{$3$};
\draw (4.5,11.5) node{$4$};
\draw (5.5,11.5) node{$5$};
\draw (6.5,11.5) node{$6$};
\draw (7.5,11.5) node{$7$};
\foreach \x in {8,9} \draw (\x,11.5) +(.5,0) node {$8$};
\foreach \x in {0,...,9} \draw (\x,12.5) +(1.5,0) node {$\x$};
\foreach \y in {0,...,10} \draw (0.5,11.5-\y) node {$\y$};
\end{tikzpicture}}
\hspace{1cm}
\subfloat[$\Delta M$]{
\begin{tikzpicture}[x=0.45cm,y=0.45cm,font=\tiny]
\clip(0,0.5) rectangle (12,13);
  \draw[style=help lines,xstep=1,ystep=1] (1,1) grid (11,12);
\foreach \x in {1,...,10} \draw (\x,1) +(.5,.5)  node {\dots};
\foreach \y in {2,...,11} \draw (10,\y) +(.5,.5) node {\dots};
\foreach \x in {1,...,8} \draw (\x,2) +(.5,.5)  node {$0$};
\foreach \y in {2,...,11} \draw (9,\y) +(.5,.5) node {$0$};
\foreach \x in {1,2} \draw (\x,3.5) +(.5,0) node{$1$};
\draw (3.5,3.5) node {$-2$};
\foreach \x in {4,...,8} \draw (\x,3.5) +(.5,0) node{$0$};
\foreach \x in {1,2} \draw (\x,4.5) +(.5,0) node{$1$};
\draw (3.5,4.5) node {$-2$};
\foreach \x in {4,...,8} \draw (\x,4.5) +(.5,0) node{$0$};
\foreach \x in {1,...,3} \draw (\x,5.5) +(.5,0) node{$1$};
\draw (4.5,5.5) node {$-3$};
\foreach \x in {5,...,8} \draw (\x,5.5) +(.5,0) node{$0$};
\foreach \x in {1,...,3} \draw (\x,6.5) +(.5,0) node{$1$};
\draw (4.5,6.5) node {$0$};
\draw (5.5,6.5) node {$-2$};
\draw (6.5,6.5) node {$-1$};
\foreach \x in {7,...,8} \draw (\x,6.5) +(.5,0) node{$0$};
\foreach \x in {1,...,4} \draw (\x,7.5) +(.5,0) node{$1$};
\draw (5.5,7.5) node {$0$};
\draw (6.5,7.5) node {$-2$};
\foreach \x in {7,8} \draw (\x,7.5) +(.5,0) node{$0$};
\foreach \x in {1,...,4} \draw (\x,8.5) +(.5,0) node{$1$};
\draw (5.5,8.5) node {$0$};
\draw (6.5,8.5) node {$0$};
\foreach \x in {7} \draw (\x,8.5) +(.5,0) node{$-1$};
\draw (8.5,8.5) node {$0$};
\foreach \x in {1,...,6} \draw (\x,9.5) +(.5,0) node{$1$};
\draw (7.5,9.5) node {$0$};
\draw (8.5,9.5) node{$-1$};
\foreach \x in {1,...,6} \draw (\x,10.5) +(.5,0) node{$1$};
\draw (7.5,10.5) node {$0$};
\draw (8.5,10.5) node {$0$};
\foreach \x in {1,...,8} \draw (\x,11.5) +(.5,0) node{$1$};
\foreach \x in {0,...,9} \draw (\x,12.5) +(1.5,0) node {$\x$};
\foreach \y in {0,...,10} \draw (0.5,11.5-\y) node {$\y$};
\end{tikzpicture}
}
\end{center}
\end{preview}
\end{figure}
We see that:
\begin{itemize}
\item $c_{27}=-1$, $a_{27}=0$ and $b_{27}=5$ and we get the point
  $P_{05}$;
\item $c_{36}=-1$, $a_{36}=0$, and $b_{36}=3$ and we get the point
  $P_{03}$;
\item $c_{45}=-2$, $a_{45}=1$ and $b_{45}=2$ and we get the points
  $P_{12}$ and $P_{23}$;
\item $c_{55}=-1$, $a_{55}=0$ and $b_{55}=0$ and we get the point
  $P_{00}$;
\item $c_{54}=-2$, $a_{54}=1$ and $b_{54}=1$ and we get the points
  $P_{11}$ and $P_{22}$;
\item $c_{63}=-3$, $a_{63}=2$ and $b_{63}=0$ and we get the points
  $P_{20}, P_{31}, P_{42}$;
\item $c_{72}=-2$, $a_{72}=5$ and $b_{72}=0$ and we get the points
  $P_{50}$ and $P_{61}$;
\item $c_{82}=-2$, $a_{82}=3$ and $b_{82}=0$ and we get the points
  $P_{30}$ and $P_{41}$.
\end{itemize}
By Theorem \ref{T:4} we have that $M$ is the Hilbert matrix of a
scheme $Z$ whose points can be represented in a grid of $(1,0)$ and
$(0,1)$-lines in the following way:
\begin{figure}[H]
\begin{preview}
\begin{center}
\begin{tikzpicture}[line cap=round,line join=round,>=triangle 45,x=0.3cm,y=0.3cm]
\clip(-1,0) rectangle (11,12);
\fill [color=black] (3,1) circle (1.5pt);
\fill [color=black] (4,1) circle (1.5pt);
\fill [color=black] (3,2) circle (1.5pt);
\fill [color=black] (4,2) circle (1.5pt);
\fill [color=black] (3,3) circle (1.5pt);
\fill [color=black] (5,3) circle (1.5pt);
\fill [color=black] (4,4) circle (1.5pt);
\fill [color=black] (5,4) circle (1.5pt);
\fill [color=black] (3,5) circle (1.5pt);
\fill [color=black] (6,5) circle (1.5pt);
\fill [color=black] (5,6) circle (1.5pt);
\fill [color=black] (6,6) circle (1.5pt);
\fill [color=black] (4,7) circle (1.5pt);
\fill [color=black] (7,7) circle (1.5pt);
\fill [color=black] (8,7) circle (1.5pt);
\fill [color=black] (3,8) circle (1.5pt);
\fill [color=black] (6,8) circle (1.5pt);
\fill [color=black] (7,8) circle (1.5pt);
\fill [color=black] (8,8) circle (1.5pt);
\fill [color=black] (4,9) circle (1.5pt);
\fill [color=black] (5,9) circle (1.5pt);
\fill [color=black] (7,9) circle (1.5pt);
\fill [color=black] (9,9) circle (1.5pt);
\fill [color=black] (10,9) circle (1.5pt);
\draw (3,0.5) -- (3,9.5);
\draw[color=black] (3,10.5) node {\tiny $C_0$};
\draw (4,0.5) -- (4,9.5);
\draw[color=black] (4,10.5) node {\tiny $C_1$};
\draw (5,2.5) -- (5,9.5);
\draw[color=black] (5,10.5) node {\tiny $C_2$};
\draw (6,4.5) -- (6,9.5);
\draw[color=black] (6,10.5) node {\tiny $C_3$};
\draw (7,6.5) -- (7,9.5);
\draw[color=black] (7,10.5) node {\tiny $C_4$};
\draw (8,6.5) -- (8,9.5);
\draw[color=black] (8,10.5) node {\tiny $C_5$};
\draw (9,8.5) -- (9,9.5);
\draw[color=black] (9,10.5) node {\tiny $C_6$};
\draw (10,8.5) -- (10,9.5);
\draw[color=black] (10,10.5) node {\tiny $C_7$};
\draw (2.5,1) -- (4.5,1);
\draw[color=black] (1.5,1) node {\tiny $R_8$};
\draw (2.5,2) -- (4.5,2);
\draw[color=black] (1.5,2) node {\tiny $R_7$};
\draw (2.5,3) -- (5.5,3);
\draw[color=black] (1.5,3) node {\tiny $R_6$};
\draw (2.5,4) -- (5.5,4);
\draw[color=black] (1.5,4) node {\tiny $R_5$};
\draw (2.5,5) -- (6.5,5);
\draw[color=black] (1.5,5) node {\tiny $R_4$};
\draw (2.5,6) -- (6.5,6);
\draw[color=black] (1.5,6) node {\tiny $R_3$};
\draw (2.5,7) -- (8.5,7);
\draw[color=black] (1.5,7) node {\tiny $R_2$};
\draw (2.5,8) -- (8.5,8);
\draw[color=black] (1.5,8) node {\tiny $R_1$};
\draw (2.5,9) -- (10.5,9);
\draw[color=black] (1.5,9) node {\tiny $R_0$};
\end{tikzpicture}
\caption{The scheme $Z$}
\end{center}
\end{preview}
\end{figure}
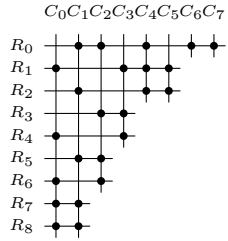
\end{Ex}

\begin{Ex} \label{Ex:0} 
In this example we make some remarks on the hypotheses of Theorem \ref{T:4}.
  \begin{enumerate}
   \item Let $M$ be a plain matrix such that either condition 1 or condition 2
of Theorem \ref{T:4} holds and suppose that it is not
  $\Delta$-regular. Then it might be
  $M_Z\ne M$. As an example let us consider a scheme  $Y$ whose points can be represented in a grid of $(1,0)$ and $(0,1)$-lines in the following way  and the associated Hilbert matrix $M=M_Y$ which satisfies the previous conditions:
\begin{figure}[H]
\begin{preview}
\begin{center}
\subfloat[$Y$]{
\begin{tikzpicture}[line cap=round,line join=round,>=triangle 45,x=0.3cm,y=0.3cm]
\clip(0,0) rectangle (6,6);
\fill [color=black] (1,1) circle (1.5pt);
\fill [color=black] (2,1) circle (1.5pt);
\fill [color=black] (1,2) circle (1.5pt);
\fill [color=black] (2,2) circle (1.5pt);
\fill [color=black] (3,3) circle (1.5pt);
\fill [color=black] (4,3) circle (1.5pt);
\fill [color=black] (5,3) circle (1.5pt);
\fill [color=black] (3,4) circle (1.5pt);
\fill [color=black] (4,4) circle (1.5pt);
\fill [color=black] (5,4) circle (1.5pt);
\fill [color=black] (2,5) circle (1.5pt);
\fill [color=black] (3,5) circle (1.5pt);
\fill [color=black] (4,5) circle (1.5pt);
\fill [color=black] (5,5) circle (1.5pt);
\draw (1,0.5) -- (1,5.5);
\draw (2,0.5) -- (2,5.5);
\draw (3,2.5) -- (3,5.5);
\draw (4,2.5) -- (4,5.5);
\draw (5,2.5) -- (5,5.5);
\draw (0.5,1) -- (2.5,1);
\draw (0.5,2) -- (2.5,2);
\draw (0.5,3) -- (5.5,3);
\draw (0.5,4) -- (5.5,4);
\draw (0.5,5) -- (5.5,5);
\end{tikzpicture}
}
\hspace{1cm}
\subfloat[$M=M_Y$]{
\begin{tikzpicture}[x=0.45cm,y=0.45cm,font=\tiny]
\clip(0,0) rectangle (9,9);
  \draw[style=help lines,xstep=1,ystep=1] (1,1) grid (8,8);
\foreach \x in {1,...,7} \draw (\x,1) +(.5,.5)  node {\dots};
\foreach \y in {2,...,7} \draw (7,\y) +(.5,.5) node {\dots};
\draw (1.5,2.5) node{$5$};
\draw (2.5,2.5) node{$10$};
\draw (3.5,2.5) node{$13$};
\foreach \x in {4,...,6} \draw (\x,2) +(.5,.5)  node {$14$};
\draw (1.5,3.5) node{$5$};
\draw (2.5,3.5) node{$10$};
\draw (3.5,3.5) node{$13$};
\foreach \x in {4,...,6} \draw (\x,3) +(.5,.5)  node {$14$};
\draw (1.5,4.5) node{$4$};
\draw (2.5,4.5) node{$8$};
\draw (3.5,4.5) node{$11$};
\draw (4.5,4.5) node{$13$};
\foreach \x in {5,6} \draw (\x,4) +(.5,.5)  node {$14$};
\draw (1.5,5.5) node{$3$};
\draw (2.5,5.5) node{$6$};
\draw (3.5,5.5) node{$9$};
\draw (4.5,5.5) node{$12$};
\foreach \x in {5,6} \draw (\x,5) +(.5,.5)  node {$14$};
\draw (1.5,6.5) node{$2$};
\draw (2.5,6.5) node{$4$};
\draw (3.5,6.5) node{$6$};
\draw (4.5,6.5) node{$8$};
\foreach \x in {5,6} \draw (\x,6) +(.5,.5)  node {$10$};
\draw (1.5,7.5) node{$1$};
\draw (2.5,7.5) node{$2$};
\draw (3.5,7.5) node{$3$};
\draw (4.5,7.5) node{$4$};
\foreach \x in {5,6} \draw (\x,7) +(.5,.5)  node {$5$};
\foreach \x in {0,...,6} \draw (\x,8.5) +(1.5,0) node {$\x$};
\foreach \y in {0,...,6} \draw (0.5,7.5-\y) node {$\y$};
\end{tikzpicture}
}
\end{center}
\end{preview}
\end{figure}
Indeed, we can take, as in the definition of $\Delta$-regular matrix,
$(i_1,j_1)=(4,3)$ and $i_2,j_2
=(3,4)$. So $c_{43}=c_{34}=-1$, $a_{43}=a_{34}=1$, while $b_{43}=1$
and $b_{34}=0$. This means that $b_{43}-b_{34}=1>0=a_{43}-a_{34}$.
Then, by adding the points on the $(1,0)$-lines and using
\cite[Theorem 3.1]{BM}, it is possible to see that $M_Z\ne M$:
\begin{figure}[H]
\begin{preview}
\begin{center}
\subfloat[$Z$]{
\begin{tikzpicture}[line cap=round,line join=round,>=triangle 45,x=0.3cm,y=0.3cm]
\clip(0,0) rectangle (6,6);
\fill [color=black] (1,1) circle (1.5pt);
\fill [color=black] (2,1) circle (1.5pt);
\fill [color=black] (1,2) circle (1.5pt);
\fill [color=black] (2,2) circle (1.5pt);
\fill [color=black] (1,3) circle (1.5pt);
\fill [color=black] (3,3) circle (1.5pt);
\fill [color=black] (4,3) circle (1.5pt);
\fill [color=black] (3,4) circle (1.5pt);
\fill [color=black] (4,4) circle (1.5pt);
\fill [color=black] (5,4) circle (1.5pt);
\fill [color=black] (2,5) circle (1.5pt);
\fill [color=black] (3,5) circle (1.5pt);
\fill [color=black] (4,5) circle (1.5pt);
\fill [color=black] (5,5) circle (1.5pt);
\draw (1,0.5) -- (1,5.5);
\draw (2,0.5) -- (2,5.5);
\draw (3,2.5) -- (3,5.5);
\draw (4,2.5) -- (4,5.5);
\draw (5,3.5) -- (5,5.5);
\draw (0.5,1) -- (2.5,1);
\draw (0.5,2) -- (2.5,2);
\draw (0.5,3) -- (4.5,3);
\draw (0.5,4) -- (5.5,4);
\draw (0.5,5) -- (5.5,5);
\end{tikzpicture}
}
\hspace{1cm}
\subfloat[$M_Z$]{
\begin{tikzpicture}[x=0.45cm,y=0.45cm,font=\tiny]
\clip(0,0) rectangle (9,9);
  \draw[style=help lines,xstep=1,ystep=1] (1,1) grid (8,8);
\foreach \x in {1,...,7} \draw (\x,1) +(.5,.5)  node {\dots};
\foreach \y in {2,...,7} \draw (7,\y) +(.5,.5) node {\dots};
\draw (1.5,2.5) node{$5$};
\draw (2.5,2.5) node{$10$};
\draw (3.5,2.5) node{$13$};
\foreach \x in {4,...,6} \draw (\x,2) +(.5,.5)  node {$14$};
\draw (1.5,3.5) node{$5$};
\draw (2.5,3.5) node{$10$};
\draw (3.5,3.5) node{$13$};
\foreach \x in {4,...,6} \draw (\x,3) +(.5,.5)  node {$14$};
\draw (1.5,4.5) node{$4$};
\draw (2.5,4.5) node{$8$};
\draw (3.5,4.5) node{$11$};
\draw (4.5,4.5) node{$14$};
\foreach \x in {5,6} \draw (\x,4) +(.5,.5)  node {$14$};
\draw (1.5,5.5) node{$3$};
\draw (2.5,5.5) node{$6$};
\draw (3.5,5.5) node{$9$};
\draw (4.5,5.5) node{$12$};
\foreach \x in {5,6} \draw (\x,5) +(.5,.5)  node {$14$};
\draw (1.5,6.5) node{$2$};
\draw (2.5,6.5) node{$4$};
\draw (3.5,6.5) node{$6$};
\draw (4.5,6.5) node{$8$};
\foreach \x in {5,6} \draw (\x,6) +(.5,.5)  node {$10$};
\draw (1.5,7.5) node{$1$};
\draw (2.5,7.5) node{$2$};
\draw (3.5,7.5) node{$3$};
\draw (4.5,7.5) node{$4$};
\foreach \x in {5,6} \draw (\x,7) +(.5,.5)  node {$5$};
\foreach \x in {0,...,6} \draw (\x,8.5) +(1.5,0) node {$\x$};
\foreach \y in {0,...,6} \draw (0.5,7.5-\y) node {$\y$};
\end{tikzpicture}
}
\end{center}
\end{preview}
\end{figure}
  \item It is easy to see that the Hilbert matrix of $3$ generic points of $\mathbb P^1\times
  \mathbb P^1$ is such that $a_{ij}\ge a_{i-1j+1}$ and $b_{ij}\ge
  b_{i+1j-1}$ for any $i,j\ge 0$ and that it is $\Delta$-regular, but it
  is not plain:
  \begin{figure}[H]
\begin{preview}
    \begin{center}
    \begin{tikzpicture}[x=0.45cm,y=0.45cm,font=\tiny]
\clip(0,1) rectangle (5,6);
  \draw[style=help lines,xstep=1,ystep=1] (1,1) grid (5,5);
\foreach \x in {1,...,4} \draw (\x,1) +(.5,.5)  node {\dots};
\foreach \y in {2,...,4} \draw (4,\y) +(.5,.5) node {\dots};
\draw (1.5,2.5) node{$3$};
\draw (2.5,2.5) node{$3$};
\draw (3.5,2.5) node{$3$};
\draw (1.5,3.5) node{$2$};
\draw (2.5,3.5) node{$3$};
\draw (3.5,3.5) node{$3$};
\draw (1.5,4.5) node{$1$};
\draw (2.5,4.5) node{$2$};
\draw (3.5,4.5) node{$3$};
\foreach \x in {0,...,3} \draw (\x,5.5) +(1.5,0) node {$\x$};
\foreach \y in {0,...,3} \draw (0.5,4.5-\y) node {$\y$};
\end{tikzpicture}
\end{center}
\end{preview}
  \end{figure}
Indeed, $a_{12}=a_{21}=0$ and $b_{12}=b_{21}=0$. In this case it is
clear that $M_Z\ne M$, because $\deg Z=4\ne 3$.
   \item Let $X\subset \mathbb P^1\times \mathbb P^1$ be a reduced
     zero-dimensional scheme whose points can be represented on a grid
     of $(1,0)$ and $(0,1)$-lines in the following way:
\begin{figure}[H]
\begin{center}
\begin{preview}
\subfloat[$X$]{
\begin{tikzpicture}[line cap=round,line join=round,>=triangle 45,x=0.3cm,y=0.3cm]
\clip(0,0) rectangle (5,5);
\fill [color=black] (1,1) circle (1.5pt);
\fill [color=black] (2,1) circle (1.5pt);
\fill [color=black] (1,2) circle (1.5pt);
\fill [color=black] (2,2) circle (1.5pt);
\fill [color=black] (3,3) circle (1.5pt);
\fill [color=black] (4,3) circle (1.5pt);
\fill [color=black] (3,4) circle (1.5pt);
\fill [color=black] (4,4) circle (1.5pt);
\draw (1,0.5) -- (1,4.5);
\draw (2,0.5) -- (2,4.5);
\draw (3,2.5) -- (3,4.5);
\draw (4,2.5) -- (4,4.5);
\draw (0.5,1) -- (2.5,1);
\draw (0.5,2) -- (2.5,2);
\draw (0.5,3) -- (4.5,3);
\draw (0.5,4) -- (4.5,4);
\end{tikzpicture}
}
\hspace{1cm}
\subfloat[$M_X$]{
\begin{tikzpicture}[x=0.45cm,y=0.45cm,font=\tiny]
\clip(0,0) rectangle (8,8);
  \draw[style=help lines,xstep=1,ystep=1] (1,1) grid (7,7);
\foreach \x in {1,...,6} \draw (\x,1) +(.5,.5)  node {\dots};
\foreach \y in {2,...,6} \draw (6,\y) +(.5,.5) node {\dots};
\draw (1.5,2.5) node{$4$};
\foreach \x in {2,...,5} \draw (\x,2) +(.5,.5)  node {$8$};
\draw (1.5,3.5) node{$4$};
\foreach \x in {2,...,5} \draw (\x,3) +(.5,.5)  node {$8$};
\draw (1.5,4.5) node{$3$};
\draw (2.5,4.5) node{$6$};
\draw (3.5,4.5) node{$7$};
\draw (4.5,4.5) node{$8$};
\draw (5.5,4.5) node{$8$};
\draw (1.5,5.5) node{$2$};
\draw (2.5,5.5) node{$4$};
\draw (3.5,5.5) node{$6$};
\draw (4.5,5.5) node{$8$};
\draw (5.5,5.5) node{$8$};
\draw (1.5,6.5) node{$1$};
\draw (2.5,6.5) node{$2$};
\draw (3.5,6.5) node{$3$};
\draw (4.5,6.5) node{$4$};
\draw (5.5,6.5) node{$4$};
\foreach \x in {0,...,5} \draw (\x,7.5) +(1.5,0) node {$\x$};
\foreach \y in {0,...,5} \draw (0.5,6.5-\y) node {$\y$};
\end{tikzpicture}
}
\end{preview}
\end{center}
\end{figure}
Then it is easy to see that the Hilbert matrix $M_X$ of $X$ is plain
and $\Delta$-regular, but it does not satisfies either condition 1
or condition 2 of Theorem \ref{T:4}. Indeed, $a_{22}=1<2=a_{13}$ and
$b_{22}=1<2=b_{31}$. However, in this case $Z=X$.
 \end{enumerate}
\end{Ex}

\begin{op}
  Given an admissible matrix $M$, which is plain and $\Delta$-regular,
  but which does not satisfy either condition 1 or condition 2, is $M$
  the Hilbert function of some zero-dimensional schemes? In particular, given the associated scheme $Z$, $M_Z=M$?   
\end{op}


\begin{thebibliography}{90}

\bibitem{BM}
P. Bonacini, L. Marino, \emph{On the Hilbert function in $\mathbb
  P^1\times \mathbb P^1$}, Collectanea Math., \textbf{62} (2011),
no. 2, 57--67.

\bibitem{GMR}
S. Giuffrida, R. Maggioni, A. Ragusa, \emph{On the postulation of
  $0$-dimensional subschemes on a smooth quadric}, Pacific
J. Math. \textbf{155} (1992), no. 2, 251--282.

\bibitem{G}
E. Guardo, \emph{Fat points schemes on a smooth quadric}, J. Pure Appl. Algebra \textbf{162} (2001), no. 2--3, 183–-208. 

\bibitem{GVT}
E. Guardo, A. Van Tuyl, \emph{Fat points in $\mathbb P^1\times \mathbb
P^1$ and their Hilbert functions}, Canad. J. Math. \textbf{56} (2004),
no. 4, 716--741.

\bibitem{GVT3}
E. Guardo, A. Van Tuyl, \emph{The minimal resolutions of double points
  in $\mathbb P^1\times \mathbb P^1$ with ACM support}, J. Pure Appl. Algebra \textbf{211} (2007), no. 3, 784–-800. 

\bibitem{GVT2}
E. Guardo, A. Van Tuyl, \emph{Separators of points in a
  multiprojective space}, Manuscripta Math.   \textbf{126} (2008),
no. 1, 99--113. 

\bibitem{M}
L. Marino, \emph{Conductor and separating degrees for sets of points
  in $\mathbb P^r$ and $\mathbb P^1\times \mathbb P^1$}, Boll. Unione
Mat. Ital. Sez. B Artic. Ric. Mat. (8) \textbf{9} (2006), no. 2, 397--421.

\bibitem{O}
F. Orecchia, \emph{Points in generic position and conductors of curves
with ordinary singularities}, J. London Math. Soc. (2), \textbf{24}
(1981), 85--96. 

\bibitem{VT}
A. Van Tuyl, \emph{An appendix to a paper of M. V. Catalisano,
  A. V. Geramita and A. Gimigliano. The Hilbert function of generic
  sets of 2-fat points in $\mathbb P^1\times\mathbb P^1$: Higher
  secant varieties of Segre-Veronese varieties, in Projective
  varieties with unexpected properties, 81--107, Walter de Gruyter
  GmbH \& Co. KG, Berlin, 2005},  Projective varieties with unexpected properties, 109–-112, Walter de Gruyter GmbH \& Co. KG, Berlin, 2005.  

\end{thebibliography}
\end{document}